\documentclass[11pt]{article}
\title{Excluding four-edge paths and their complements}
\author{Maria Chudnovsky \thanks{Columbia University, New York, NY 10027, USA. E-mail: mchudnov@columbia.edu. Partially supported by NSF grants DMS-1001091
and IIS-1117631.}, Peter Maceli \thanks{Columbia University, New York, NY 10027, USA. E-mail: plm2109@columbia.edu.}, and Irena Penev \thanks{Park Tudor School, Indianapolis, IN 46240, USA. E-mail: ip2158@caa.columbia.edu.}} 
\date{\today} 
\usepackage{amsthm}
\usepackage{amssymb}
\usepackage{amsbsy}
\usepackage{mathrsfs} 
\newtheorem{theorem}{}[section]
\begin{document} 
\maketitle 
\noindent 

\begin{abstract} 
\noindent 
We prove that a graph $G$ contains no induced four-edge path and no induced complement of a four-edge path if and only if $G$ is obtained from five-cycles and split graphs by repeatedly applying the following operations: substitution, split graph unification, and split graph unification in the complement (``split graph unification'' is a new class-preserving operation that is introduced in this paper).
\end{abstract}

\section{Introduction} 

All graphs in this paper are finite and simple. We denote by $P_5$ the path on five vertices (and four edges); this path is also called a {\em four-edge path}. The complement of a graph $G$ is denoted by $\overline{G}$. The graph $\overline{P_5}$, the complement of the four-edge path, is also called a {\em house}. Given graphs $G$ and $H$, we say that $G$ is {\em $H$-free} if $G$ does not contains (an isomorphic copy of) $H$ as an induced subgraph. Given a family $\mathcal{H}$ of graphs, we say that a graph $G$ is {\em $\mathcal{H}$-free} provided that $G$ is $H$-free for all $H \in \mathcal{H}$. The goal of this paper is to understand the structure of $\{P_5,\overline{P_5}\}$-free graphs. 
\\
\\
We begin with a few definitions. The vertex-set of a graph $G$ is denoted by $V_G$. Given $X \subseteq V_G$, we denote by $G[X]$ the subgraph of $G$ induced by $X$; given $v_1,...,v_n \in V_G$, we often write $G[v_1,...,v_n]$ instead of $G[\{v_1,...,v_n\}]$. We denote by $G \smallsetminus X$ the graph $G[V_G \smallsetminus X]$, and for $v \in V_G$, we often write $G \smallsetminus v$ instead of $G \smallsetminus \{v\}$. A {\em clique} in $G$ is a set of pairwise adjacent vertices in $G$, and the {\em clique number} of $G$, denoted by $\omega(G)$, is the maximum size of a clique in $G$. A {\em stable set} in $G$ is a set of pairwise non-adjacent vertices in $G$. $G$ is said to be a {\em split graph} if its vertex-set can be partitioned into a (possibly empty) clique and a (possibly empty) empty set. The chromatic number of $G$ is denoted by $\chi(G)$. $G$ is said to be {\em perfect} if $\chi(H)=\omega(H)$ for all induced subgraphs $H$ of $G$. Given a graph $G$, a set $X \subseteq V_G$, and a vertex $v \in V_G \smallsetminus X$, we say that $v$ is {\em complete} to $X$ if $v$ is adjacent to every vertex of $X$, and we say that $v$ is {\em anti-complete} to $X$ if $v$ is non-adjacent to every vertex of $X$; $v$ is said to be {\em mixed} on $X$ if $v$ is neither complete nor anti-complete to $X$. $X$ is said to be a {\em homogeneous set} in $G$ if no vertex in $V_G \smallsetminus X$ is mixed on $X$. A homogeneous set $X$ in a graph $G$ is said to be {\em proper} if $2 \leq |X| \leq |V_G|-1$. $G$ is said to be {\em prime} if it does not contain a proper homogeneous set. 
\\
\\
We denote by $C_5$ the cycle on five vertices; this graph is also called a {\em pentagon}. The following result about the structure of $\{P_5,\overline{P_5}\}$-free graphs was proven by Fouquet in \cite{Fouquet}. 
\begin{theorem} \cite{Fouquet} \label{Fouquet} 
For each $\{P_5,\overline{P_5}\}$-free graph G at least one of the following 
holds:
\begin{itemize} 
\item $G$ contains a proper homogeneous set; 
\item $G$ is isomorphic to $C_5$; 
\item $G$ is $C_5$-free. 
\end{itemize} 
\end{theorem} 
\noindent 
\ref{Fouquet} immediately implies that every $\{P_5,\overline{P_5},C_5\}$-free graph can be obtained by ``substitution'' starting from $\{P_5,\overline{P_5},C_5\}$-free graphs and pentagons (substitution is a well-known operation whose precise definition we give in section \ref{section:def}). Furthermore, it is easy to check that every graph obtained by substitution starting from $\{P_5,\overline{P_5},C_5\}$-free graphs and pentagons is $\{P_5,\overline{P_5}\}$-free. We remark that the Strong Perfect Graph Theorem \cite{SPGT} implies that a $\{P_5,\overline{P_5}\}$-free graph is perfect if and only if it is $C_5$-free. Thus, every $\{P_5,\overline{P_5}\}$-free graph can be obtained by substitution starting from $\{P_5,\overline{P_5}\}$-free perfect graphs and pentagons. In view of this, the bulk of this paper focuses on $\{P_5,\overline{P_5},C_5\}$-free graphs (equivalently: $\{P_5,\overline{P_5}\}$-free perfect graphs). 
\\
\\
It is easy to check that all split graphs are $\{P_5,\overline{P_5},C_5\}$-free. Our first result is a decomposition theorem (\ref{skew-partition-two-nontrivial}), which states that every prime $\{P_5,\overline{P_5},C_5\}$-free graph that is not split admits a particular kind of ``skew-partition.'' Skew-partitions were first introduced by Chv{\'a}tal \cite{chvatal}, and they played an important role in the proof of the Strong Perfect Graph Theorem \cite{SPGT}; we give the precise definition in section \ref{section:skew-partition}. Our second result is another decomposition theorem (\ref{decomposition}), which states that every prime $\{P_5,\overline{P_5},C_5\}$-free graph that is not split admits a new kind of decomposition, which we call a ``split graph divide'' (see section \ref{section:split-graph-divide} for the definition). Next, we reverse the split graph divide decomposition to turn it into a composition that preserves the property of being $\{P_5,\overline{P_5},C_5\}$-free. We call this composition ``split graph unification'' (see section \ref{section:split-graph-unification} for the definition). Finally, combining our results with \ref{Fouquet}, we prove that every $\{P_5,\overline{P_5}\}$-free graph is obtained by repeatedly applying substitution, split graph unification, and split graph unification in the complement starting from split graphs and pentagons (see \ref{P5-P5c} and \ref{cor}). 
\\
\\
This paper is organized as follows. Section \ref{section:def} contains definitions that we use in the remainder of the paper. Section \ref{section:skew-partition} is devoted, first, to proving that every prime $\{P_5,\overline{P_5},C_5\}$-free graph that is not split admits a skew-partition of a certain kind (see \ref{skew-partition-two-nontrivial}), and then to further analyzing skew-partitions in prime $\{P_5,\overline{P_5},C_5\}$-free graphs. The final result of section \ref{section:skew-partition} (see \ref{skew-partition}) is used in section \ref{section:split-graph-divide}. However, a number of lemmas from section \ref{section:skew-partition} (in particular \ref{preliminary-decomposition}, \ref{not-mixed-on-other-half}, and \ref{skew-partition analysis}) that are used to prove \ref{skew-partition} may also be of independent interest as theorems about skew-partitions in $\{P_5,\overline{P_5},C_5\}$-free graphs. Section \ref{section:split-graph-divide} deals with split graph divides, and section \ref{section:split-graph-unification} with split graph unifications. Finally, in section \ref{section:main-thm}, we prove the main theorem of this paper.

\section{Definitions} \label{section:def} 

Given a graph $G$ and a vertex $v \in V_G$, we denote by $\Gamma_G(v)$ the set of all neighbors of $v$ in $G$. (Thus, $v \notin \Gamma_G(v)$.) The {\em degree} of $v$ in $G$ is $|\Gamma_G(v)|$, that is, the number of neighbors that $v$ has in $G$. 
\\
\\
A graph is {\em non-trivial} if it contains at least two vertices. A graph $H$ is said to be {\em smaller} than a graph $G$ provided that $H$ has strictly fewer vertices than $G$. $G$ is {\em bigger} than $H$ provided that $H$ is smaller than $G$. 
\\
\\
Given graphs $G_1$ and $G_2$ with disjoint vertex-sets, and a vertex $u \in V_{G_2}$, we say that a graph $G$ is obtained by {\em substituting $G_1$ for $u$ in $G_2$} provided that the following hold: 
\begin{itemize} 
\item $V_G = V_{G_1} \cup (V_{G_2} \smallsetminus \{u\})$; 
\item $G[V_{G_1}] = G_1$; 
\item $G[V_{G_2} \smallsetminus \{u\}] = G_2 \smallsetminus u$; 
\item for all $v \in V_{G_2} \smallsetminus \{u\}$, if $v$ is adjacent to $u$ in $G_2$, then $v$ is complete to $V_{G_1}$ in $G$, and if $v$ is non-adjacent to $u$ in $G_2$, then $v$ is anti-complete to $V_{G_1}$ in $G$. 
\end{itemize} 
\noindent 
We remark that under these circumstances, $V_{G_1}$ is a homogeneous set in $G$, and the homogeneous set $V_{G_1}$ in $G$ is proper if and only if $G_1$ and $G_2$ both have at least two vertices (equivalently: if $G_1$ and $G_2$ are both smaller than $G$). Thus, a graph $G$ is obtained by substitution from smaller graphs if and only if $G$ contains a proper homogeneous set. 
\\
\\
Given a graph $G$ and disjoint sets $A,B \subseteq V_G$, we say that $A$ is {\em complete} to $B$ provided that every vertex in $A$ is complete to $B$, and we say that $A$ is {\em anti-complete} to $B$ provided every vertex in $A$ is anti-complete to $B$. 
\\
\\
We often denote a path by $p_0-...-p_n$; this means that $p_0,...,p_n$ are the vertices of the path, and that for all distinct $i,j \in \{0,...,n\}$, $p_i$ is adjacent to $p_j$ if and only if $|i-j| = 1$. A path on $n+1$ vertices and $n$ edges is denoted by $P_{n+1}$; thus, $P_{n+1}$ is an $n$-edge path. The {\em length} of a path is the number of edges that it contains; thus, the length of $P_{n+1}$ is $n$. We remind the reader that a house is a the complement of a four-edge path. We often denote a house by $p_0-p_1-p_2-p_3-p_4$; this means that $p_0,p_1,p_2,p_3,p_4$ are the vertices of the house, and that for all distinct $i,j \in \{0,1,2,3,4\}$, $p_i$ and $p_j$ are non-adjacent if and only if $|i-j| = 1$. 
\\
\\
We often denote a cycle by $c_0-c_1-...-c_{n-1}-c_0$; this means that $c_0,c_1,...,c_{n-1}$ are the vertices of the cycle, and that for all distinct $i,j \in \{0,...,n-1\}$, $c_i$ and $c_j$ are adjacent if and only if $|i-j|=1$ or $n-1$. A cycle on $n$ vertices and $n$ edges is denoted by $C_n$. The {\em length} of a cycle is the number of edges (equivalently: the number of vertices) that it contains. A {\em triangle} is a cycle of length three, a {\em square} is a cycle of length four, and a {\em pentagon} is a cycle of length five. 
\\
\\
A graph $G$ is {\em connected} if $V_G$ cannot be partitioned into two non-empty sets that are anti-complete to each other. A graph $G$ is {\em anti-connected} if $\overline{G}$ is connected. A {\em component} of a non-null graph $G$ is a maximal connected induced subgraph of $G$, and an {\em anti-component} of $G$ is a maximal anti-connected induced subgraph of $G$. A component or an anti-component of a graph is {\em non-trivial} if it contains at least two vertices. We remark that the vertex-sets of the components of a graph are anti-complete to each other, and that the vertex-sets of the anti-components of a graph are complete to each other.

\section{Skew-partition decomposition} \label{section:skew-partition} 

Given a graph $G$ and sets $X,Y \subseteq V_G$, we say that $(X,Y)$ is a {\em skew-partition} of $G$ provided that $V_G = X \cup Y$, $X$ and $Y$ are non-empty and disjoint, $G[X]$ is not connected, and $G[Y]$ is not anti-connected. We say that $G$ {\em admits a skew-partition} provided that there exist sets $X,Y \subseteq V_G$ such that $(X,Y)$ is a skew-partition of $G$. Note that if $(X,Y)$ is a skew-partition of $G$, then $(Y,X)$ is a skew-partition of $\overline{G}$, and consequently, $G$ admits a skew-partition if and only if $\overline{G}$ does. 
\\
\\
This section is devoted to analyzing skew-partitions in $\{P_5,\overline{P_5},C_5\}$-free graphs. It is organized as follows. We consider prime $\{P_5,\overline{P_5},C_5\}$-free graphs that are not split. A result from \cite{MariaPeter} (see \ref{MariaPeter} below) guarantees that all such graphs contain a certain induced subgraph that ``interacts'' with the rest of the graph in a certain useful way. We use this result to show that every prime $\{P_5,\overline{P_5},C_5\}$-free graph $G$ that is not split admits a skew-partition $(X,Y)$ such that either $G[X]$ contains at least two non-trivial components, or $G[Y]$ contains at least two non-trivial anti-components (see \ref{skew-partition-two-nontrivial}). The remainder of the section is devoted to proving a series of lemmas about skew-partitions of this kind in $\{P_5,\overline{P_5},C_5\}$-free graphs. The final result of this section (\ref{skew-partition}) is used in section \ref{section:split-graph-divide} to construct ``split graph divides.'' 
\\
\\
We begin with a couple of technical lemmas (\ref{mixed} and \ref{not-mixed-on-two}). 
\begin{theorem} \label{mixed} Let $G$ be a graph, let $X \subseteq V_G$, and let $v \in V_G \smallsetminus X$ be a vertex that is mixed on $X$ in $G$. Then the following hold: 
\begin{itemize}
\item if $G[X]$ is connected, then there exist adjacent $x,x' \in X$ such that $v$ is adjacent to $x$ and non-adjacent to $x'$; 
\item if $G[X]$ is anti-connected, then there exist non-adjacent $x,x' \in X$ such that $v$ is adjacent to $x$ and non-adjacent to $x'$.
\end{itemize}
\end{theorem}
\begin{proof}
It suffices to prove the first statement, for then the second will follow by an analogous argument applied to $\overline{G}$. So suppose that $G[X]$ is connected. Since $v$ is mixed on $X$, both $X \cap \Gamma_G(v)$ and $X \smallsetminus \Gamma_G(v)$ are non-empty. As $G[X]$ is connected, $X \cap \Gamma_G(v)$ is not anti-complete to $X \smallsetminus \Gamma_G(v)$ in $G$, and consequently, there exist adjacent verices $x \in X \cap \Gamma_G(v)$ and $x' \in X \smallsetminus \Gamma_G(v)$. This completes the argument. 
\end{proof}
\begin{theorem} \label{not-mixed-on-two} Let $G$ be a $\{P_5,\overline{P_5}\}$-free graph, and let $(X,Y)$ be a skew-partition of $G$. Then: 
\begin{itemize} 
\item no vertex in $X$ is mixed on more than one anti-component of $G[Y]$; 
\item no vertex in $Y$ is mixed on more than one component of $G[X]$. 
\end{itemize} 
\end{theorem} 
\begin{proof} 
It suffices to prove the second statement, for the first will then follow by an analogous argument applied to $\overline{G}$. Suppose that $X_1$ and $X_2$ are the vertex-sets of distinct components of $G[X]$, and that some $y \in Y$ is mixed on both $X_1$ and $X_2$ in $G$. By \ref{mixed}, there exist adjacent $x_1,x_1' \in X_1$ such that $y$ is adjacent to $x_1$ and non-adjacent to $x_1'$. Similarly, there exist adjacent $x_2,x_2' \in X_2$ such that $y$ is adjacent to $x_2$ and non-adjacent to $x_2'$. But now $x_1'-x_1-y-x_2-x_2'$ is an induced four-edge path in $G$, which contradicts the assumption that $G$ is $P_5$-free. 
\end{proof} 
\noindent 
We now need some definitions. Given a graph $G$ and a vertex $v \in V_G$, we say that $v$ is a {\em simplicial} vertex of $G$ provided that $\Gamma_G(v)$ is a clique, and we say that $v$ is an {\em anti-simplicial} vertex of $G$ provided that $v$ is a simplicial vertex of $\overline{G}$. In other words, $v$ is simplicial in $G$ provided that $v$'s neighbors in $G$ form a clique, and $v$ is anti-simplicial in $G$ provided that $v$'s non-neighbors in $G$ form a stable set. Last, $H_6$ is a graph with vertex-set $\{v_1,v_2,v_3,$ $v_4,v_5,v_6\}$ and edge-set $\{v_1v_2,v_2v_3,v_3v_4,v_2v_5,v_3v_6,v_5v_6\}$. 
\\
\\
In what follows, we use a theorem (stated below) proven in \cite{MariaPeter} by the first two authors of the present paper. 
\begin{theorem} \cite{MariaPeter} \label{MariaPeter} If $G$ is a prime $\{P_5,\overline{P_5},C_5\}$-free graph, then either $G$ is a split graph or at least one of $G$ and $\overline{G}$ contains an induced $H_6$ whose two vertices of degree one are simplicial, and at least one of whose vertices of degree three is anti-simplicial.
\end{theorem} 
\noindent 
Given a graph $G$, an induced subgraph $H$ of $G$, and vertices $u \in V_H$ and $u' \in V_G \smallsetminus V_H$, we say that $u'$ is a {\em clone} of $u$ with respect to $H$ in $G$ provided that for all $v \in V_H \smallsetminus \{u\}$, $u'$ is adjacent to $v$ if and only if $u$ is adjacent to $v$. We now prove a lemma (see \ref{three-edge-path}) that describes how vertices in a $\{P_5,\overline{P_5},C_5\}$-free graph can ``attach'' to an induced three-edge path in that graph, and then we use this lemma, together with \ref{MariaPeter}, to show that every prime $\{P_5,\overline{P_5},C_5\}$-free graph that is not split admits a certain kind of skew-partition (see \ref{skew-partition-two-nontrivial}). 

\begin{theorem} \label{three-edge-path} Let $G$ be a $\{P_5,\overline{P_5},C_5\}$-free graph, and let $a-b-c-d$ be an induced three-edge path in $G$. For each $x \in \{a,b,c,d\}$, let $C_x$ be the set of all clones of $x$ with respect to $a-b-c-d$ in $G$. Let $A$ be the set of all vertices in $G$ that are anti-complete to $\{a,b,c,d\}$, let $B$ be the set of all vertices in $G$ that are complete to $\{b,c\}$ and anti-complete to $\{a,d\}$, and let $C$ be the set of vertices in $G$ that are complete to $\{a,b,c,d\}$. Then $V_G = \{a,b,c,d\} \cup C_a \cup C_b \cup C_c \cup C_d \cup A \cup B \cup C$. 
\end{theorem} 
\begin{proof} 
Let $x \in V_G$. We need to show that $x \in \{a,b,c,d\} \cup C_a \cup C_b \cup C_c \cup C_d \cup A \cup B \cup C$. If $x \in \{a,b,c,d\}$, then we are done; so assume that $x \in V_G \smallsetminus \{a,b,c,d\}$. We remark that both the premises and the conclusion of \ref{three-edge-path} are complement-invariant (we are using the fact that the complement of a three edge-path is again a three-edge path). Thus, passing to the complement of $G$ if necessary, we may assume that $x$ has at most two neighbors in $\{a,b,c,d\}$. If $x$ is anti-complete to $\{a,b,c,d\}$, then $x \in A$, and we are done. So assume that $x$ has a neighbor in $\{a,b,c,d\}$. 
\\
\\
Suppose first that $x$ has a neighbor in $\{a,d\}$; by symmetry, we may assume that $x$ is adjacent to $a$. Since $x-a-b-c-d$ is not an induced four-edge path in $G$, it follows that $x$ has a (unique) neighbor in $\{b,c,d\}$. Since $x-a-b-c-d-x$ is not an induced pentagon in $G$, $x$ is non-adjacent to $d$. But now, if $x$ is adjacent to $b$, then $x \in C_a$; and if $x$ is adjacent to $c$, then $x \in C_b$. In either case, we are done.
\\
\\
Suppose now that $x$ is anti-complete to $\{a,d\}$. If $x$ is adjacent to exactly one of $b$ and $c$, then $x \in C_a \cup C_d$; and if $x$ is complete to $\{b,c\}$, then $x \in B$. This completes the argument.
\end{proof} 

\begin{theorem} \label{skew-partition-two-nontrivial} Let $G$ be a prime $\{P_5,\overline{P_5},C_5\}$-free graph that is not split. Then $G$ admits a skew-partition $(X,Y)$ such that either $G[X]$ has at least two non-trivial components, or $G[Y]$ has at least two non-trivial anti-components. 
\end{theorem} 
\begin{proof} 
By \ref{MariaPeter}, we know that at least one of $G$ and $\overline{G}$ contains an induced $H_6$ whose two vertices of degree one are simplicial, and at least one of whose vertices of degree three is anti-simplicial. Our goal is to prove the following: 
\begin{itemize} 
\item[(a)] if $G$ contains an induced $H_6$ whose two vertices of degree one are simplicial, and at least one of whose vertices of degree three is anti-simplicial, then $G$ contains a skew-partition $(X,Y)$ such that $G[Y]$ contains at least two non-trivial anti-components; 
\item[(b)] if $\overline{G}$ contains an induced $H_6$ whose two vertices of degree one are simplicial, and at least one of whose vertices of degree three is anti-simplicial, then $G$ contains a skew-partition $(X,Y)$ such that $G[X]$ contains at least two non-trivial components. 
\end{itemize} 
Note that it suffices to prove (b), for (a) will then follow by an analogous argument applied to $\overline{G}$. So assume that $\overline{G}$ contains an induced $H_6$ whose two vertices of degree one are simplicial, and at least one of whose vertices of degree three is anti-simplicial. 
Then there exist pairwise distinct vertices $a,b,c,d,b',c' \in V_G$ such that: 
\begin{itemize} 
\item $a-b-c-d$ is an induced path in $G$; 
\item $b'c'$ is a non-edge in $G$; 
\item $b'$ is complete to $\{a,b,c\}$ and non-adjacent to $d$ in $G$; 
\item $c'$ is complete to $\{b,c,d\}$ and non-adjacent to $a$ in $G$; 
\item $a$ is simplicial in $G$; 
\item $b$ and $c$ are anti-simplicial in $G$. 
\end{itemize} 
\noindent 
Define sets $C_a,C_b,C_c,C_d,A,B,C$ as in \ref{three-edge-path}; by \ref{three-edge-path}, we know that $V_G = \{a,b,c,d\} \cup C_a \cup C_b \cup C_c \cup C_d \cup A \cup B \cup C$. We remark that $b' \in C_b$ and $c' \in C_c$. 
\\
\\
Since $a$ is simplicial and complete to $C \cup C_b \cup \{b\}$, we know that $C \cup C_b \cup \{b\}$ is a clique. Since $b$ is anti-simplicial and anti-complete to $A \cup C_d \cup \{d\}$, we know that $A \cup C_d \cup \{d\}$ is a stable set; and since $c$ is anti-simplicial and anti-complete to $A \cup C_a \cup \{a\}$, we know that $A \cup C_a \cup \{a\}$ is a stable set. 
\\
\\
Next, we claim that $C_a \cup C_d$ is stable. Suppose otherwise. Since $C_a$ and $C_d$ are stable, there exist adjacent $\hat{a} \in C_a$ and $\hat{d} \in C_d$. Since $b-\hat{d}-b'-\hat{a}-c$ is not an induced house in $G$, $b'$ has a neighbor in $\{\hat{a},\hat{d}\}$. Similarly, $c'$ has a neighbor in $\{\hat{a},\hat{d}\}$. Let $P$ be an induced path in $G[b',c',\hat{a},\hat{d}]$ between $b'$ and $c'$; since $b'c'$ is a non-edge in $G$, we know that $P$ contains at least two edges. But now since $C_a \cup \{a\}$ and $C_d \cup \{d\}$ are stable, it follows that $a-b'-P-c'-d$ is an induced path in $G$ of length at least four, contrary to the fact that $G$ is $P_5$-free. This proves that $C_a \cup C_d$ is stable.
\\
\\
We now know the following: 
\begin{itemize} 
\item $A \cup C_d \cup \{d\}$ is stable; 
\item $A \cup C_a \cup \{a\}$ is stable; 
\item $C_a \cup C_d$ is stable; 
\item $a$ is anti-complete to $C_d \cup \{d\}$; 
\item $d$ is anti-complete to $C_a \cup \{a\}$. 
\end{itemize} 
\noindent 
Consequently, $A \cup C_a \cup C_d \cup \{a,d\}$ is a stable set. 
\\
\\
Recall that $b' \in C_b$, and so $C_b \neq \emptyset$. Let $b_1$ be a vertex in $C_b$ with as few neighbors as possible in $C_c$; let $N$ be the set of all neighbors of $b_1$ in $C_c$. We claim that $N$ is complete to $C_b$. Suppose otherwise. Fix $b_2 \in C_b$ and $c_1 \in N$ such that $b_2c_1$ is a non-edge. By the minimality of $N$, there exists some $c_2 \in C_c \smallsetminus N$ such that $b_2c_2$ is an edge. Since $C_b$ is a clique, we know that either $c_1-b_1-b_2-c_2$ is an induced three-edge path in $G$, or $c_1-b_1-b_2-c_2-c_1$ is an induced square in $G$; in the former case, $d-c_1-b_1-b_2-c_2-d$ is an induced pentagon in $G$, and in the latter case, $c_2-b_1-d-b_2-c_1$ is an induced house in $G$. But neither outcome is possible since $G$ contains no induced pentagon and no induced house. This proves that $N$ is complete to $C_b$. 
\\
\\
Set $Y = C \cup N \cup (C_b \smallsetminus \{b_1\}) \cup \{b,c\}$. By definition, $b$ is complete to $C \cup C_c \cup \{c\}$; since $N \subseteq C_c$, it follows that $b$ is complete to $C \cup N \cup \{c\}$. Further, we showed above that $C_b \cup \{b\}$ is a clique; it follows that $b$ is complete to $C_b$, and consequently, to $C_b \smallsetminus \{b_1\}$. Thus, $b$ is complete to $C \cup N \cup (C_b \smallsetminus \{b_1\}) \cup \{c\} = Y \smallsetminus \{b\}$. Since $Y \smallsetminus \{b\} \neq \emptyset$ (because $c \in Y \smallsetminus \{b\}$), it follows that $Y$ is not anti-connected. 
\\
\\
Set $X = V_G \smallsetminus Y$. Then $X = A \cup B \cup C_a \cup C_d \cup (C_c \smallsetminus N) \cup \{a,d,b_1\}$. We showed above that $A \cup C_a \cup \{a\}$ is a stable set; thus, $a$ is anti-complete to $A \cup C_a$. Further, by construction, $a$ is anti-complete to $B \cup C_d \cup C_c \cup \{d\}$. It follows that $a$ is anti-complete to $X \smallsetminus \{a,b_1\}$; since $b_1$ is a clone of $b$ for $a-b-c-d$ in $G$, we know that $ab_1$ is an edge. Next, we showed above that $A \cup C_d \cup \{d\}$ is a stable set, and by the definition of $B$ and $C_a$, $d$ is anti-complete to $B \cup C_a$. Since $a-b-c-d$ is an induced path in $G$, and since $b_1 \in C_b$, we know that $d$ is anti-complete to $\{a,b_1\}$. Further, by construction, $d$ is complete to $C_c$, and therefore, to $C_c \smallsetminus N$. It follows that $d$ is complete to $C_c \smallsetminus N$ and anti-complete to $X \smallsetminus ((C_c \smallsetminus N) \cup \{d\})$ in $G$. 
\\
\\
Now, we claim that there is no path between $\{a,b_1\}$ and $(C_c \smallsetminus N) \cup \{d\}$ in $G[X]$. Suppose otherwise. Let $P$ be a path of minimum length between $\{a,b_1\}$ and $(C_c \smallsetminus N) \cup \{d\}$ in $G[X]$. Since $b_1$ is the only neighbor of $a$ in $G[X]$, and since all the neighbors of $d$ in $G[X]$ lie in $C_c \smallsetminus N$, it follows that the endpoints of this path are $b_1$ and some vertex $\hat{c} \in C_c \smallsetminus N$. Since $b_1$ is anti-complete to $C_c \smallsetminus N$, $P$ has at least two edges. But then $a-b_1-P-\hat{c}-d$ is an induced path in $G$ of length at least four, which is impossible since $G$ is $P_5$-free. Thus, there is no path between $\{a,b_1\}$ and $(C_c \smallsetminus N) \cup \{d\}$ in $G[X]$. It follows that $G[X]$ is disconnected, and consequently, that $(X,Y)$ is a skew-partition of $G$. 
\\
\\
It remains to show that $G[X]$ has at least two non-trivial components. Since $ab_1$ is an edge in $G$, $G[a,b_1]$ is connected, and since $d$ is complete to $C_c \smallsetminus N$ in $G$, $G[(C_c \smallsetminus N) \cup \{d\}]$ is connected. Let $X_1$ be the vertex-set of the component of $G[X]$ that contains $a$ and $b_1$, and let $X_2$ be the vertex-set of the component of $G[X]$ that includes $(C_c \smallsetminus N) \cup \{d\}$. Clearly, $|X_1| \geq 2$, and we just need to show that $|X_2| \geq 2$. It suffices to show that $C_c \smallsetminus N \neq \emptyset$. Suppose otherwise. Then $C_c = N$, and consequently, $b_1$ is complete to $C_c$. But $b' \in C_b$ and $b'$ has a non-neighbor (namely, $c'$) in $C_c$; consequently, $b'$ has fewer neighbors in $C_c$ than $b_1$ does, contrary to the choice of $b_1$. It follows that $|X_2| \geq 2$. This completes the argument. 
\end{proof} 
\noindent 
In the remainder of this section, we study skew-partitions of the kind that appears in \ref{skew-partition-two-nontrivial}. We start with a few definitions. Given a graph $G$ and a skew-partition $(X,Y)$ of $G$, a {\em decomposition of $(X,Y)$ in $G$} is an ordered six-tuple $(\{X_i\}_{i=1}^m,\{Y_j\}_{j=1}^n,S,K,\{S_j\}_{j=1}^n,\{K_i\}_{i=1}^m)$ such that: 
\begin{itemize} 
\item $X_1,...,X_m$ are the vertex-sets of the non-trivial components of $G[X]$; 
\item $Y_1,...,Y_n$ are the vertex-sets of the non-trivial anti-components of $G[Y]$; 
\item $S = X \smallsetminus (X_1 \cup ... \cup X_m)$; 
\item $K = Y \smallsetminus (Y_1 \cup ... \cup Y_n)$; 
\item for each $j \in \{1,...,n\}$, $S_j$ is the set of all vertices in $S$ that are mixed on $Y_j$; 
\item for each $i \in \{1,...,m\}$, $K_i$ is the set of all vertices in $K$ that are mixed on $X_i$. 
\end{itemize} 
\noindent 
Clearly, $X$ is the disjoint union of $X_1,...,X_m,S$; and $Y$ is the disjoint union of $Y_1,...,Y_n,K$. Further, $S$ is a (possibly empty) stable set, and $K$ is a (possibly empty) clique. We note that if $m = 0$, then $X = S$; similarly, if $n = 0$, then $Y = K$. 
\\
\\
We say that a skew-partition $(X,Y)$ of $G$ is {\em usable} provided that its associated partition $(\{X_i\}_{i=1}^m,\{Y_j\}_{j=1}^n,S,K,\{S_j\}_{j=1}^n,\{K_i\}_{i=1}^m)$ satisfies at least one of the following: 
\begin{itemize} 
\item[(a)] $m \geq 2$; the sets $S_1,...,S_n$ are pairwise disjoint; the sets $K_1,...,K_m$ are pairwise disjoint, and 
every vertex of $Y$ has a neighbor in each of $X_1,...,X_m$.
\item[(b)] $n \geq 2$; the sets $S_1,...,S_n$ are pairwise disjoint; the sets $K_1,...,K_m$ are pairwise disjoint, and 
every vertex in $X$ has a non-neighbor in each of $Y_1,...,Y_n$. 
\end{itemize} 
\noindent 
We say that $G$ {\em admits a usable skew-partition} provided that there exists a usable skew-partition $(X,Y)$ of $G$. Note that $G$ admits a usable skew-partition if and only if $\overline{G}$ does. Our next goal is to prove that every prime $\{P_5,\overline{P_5},C_5\}$-free graph that is not split admits a usable skew-partition (see \ref{preliminary-decomposition}). We first need a couple of lemmas (\ref{not-mixed-on-edge} and \ref{mixed-one-way}). The first of the two lemmas is used to prove the second, and the second is used in the proof of \ref{preliminary-decomposition}.

\begin{theorem} \label{not-mixed-on-edge} Let $G$ be a $\{P_5,\overline{P_5},C_5\}$-free graph, and let $X,Y \subseteq V_G$ be disjoint sets such that $G[X]$ is connected and $G[Y]$ is anti-connected. Let $v \in V_G \smallsetminus (X \cup Y)$ be complete to $Y$ and anti-complete to $X$. Then the following hold: 
\begin{itemize} 
\item if $y,y' \in Y$ are non-adjacent vertices such that some vertex in $X$ is adjacent to $y$ and non-adjacent to $y'$, then $y$ is complete to $X$ and $y'$ is anti-complete to $X$; 
\item if $x,x' \in X$ are adjacent vertices such that some vertex in $Y$ is adjacent to $x$ and non-adjacent to $x'$, then $x$ is complete to $Y$ and $x'$ is anti-complete to $Y$. 
\end{itemize} 
\end{theorem} 
\begin{proof} 
It suffices to prove the first claim, for then the second will follow by an analogous argument applied to $\overline{G}$. Suppose that $y,y' \in Y$ are non-adjacent vertices such that some vertex in $X$ is adjacent to $y$ and non-adjacent to $y'$. Let $X_0$ be the set of all vertices in $X$ that are adjacent to $y$ and non-adjacent to $y'$; by construction, $X_0 \neq \emptyset$. We need to show that $X_0 = X$. Suppose otherwise. Since $G[X]$ is connected, there exist adjacent vertices $x \in X_0$ and $x' \in X \smallsetminus X_0$. Since $x' \in X \smallsetminus X_0$, we know that one of the following holds: 
\begin{itemize} 
\item $x'$ is complete to $\{y,y'\}$; 
\item $x'$ is anti-complete to $\{y,y'\}$; 
\item $x'$ is adjacent to $y'$ and non-adjacent to $y$. 
\end{itemize} 
But in the first case, $x'-v-x-y'-y$ is an induced house in $G$; in the second case, $x'-x-y-v-y'$ is an induced four-edge path in $G$; and in the third case, $x-y-v-y'-x'-x$ is an induced pentagon in $G$. Since $G$ is $\{P_5,\overline{P_5},C_5\}$-free, none of these three outcomes is possible. Thus, $X_0 = X$. This completes the argument. 
\end{proof} 
\begin{theorem}
\label{mixed-one-way}
Let $G$ be a $\{P_5,\overline{P_5},C_5\}$-free graph, and let $X,Y \subseteq V_G$ be disjoint sets such that $G[X]$ is connected and $G[Y]$ is anti-connected. Let $v \in V_G \smallsetminus (X \cup Y)$ be complete to $Y$ and anti-complete to $X$. Then the following hold: 
\begin{itemize} 
\item if $x_0 \in X$ is mixed on $Y$, then $X$ is complete to $Y \cap \Gamma_G(x_0)$ and anti-complete to $Y \smallsetminus \Gamma_G(x_0)$; 
\item if $y_0 \in Y$ is mixed on $X$, then $Y$ is complete to $X \cap \Gamma_G(y_0)$ and anti-complete to $X \smallsetminus \Gamma_G(y_0)$.
\end{itemize}
\end{theorem}
\begin{proof}
It suffices to prove the first claim, for then the second will follow by an analogous argument applied to $\overline{G}$. Suppose that some $x_0 \in X$ is mixed on $Y$, and let $U = Y \cap \Gamma_G(x_0)$ and $V = Y \smallsetminus \Gamma_G(x_0)$. Then $U$ and $V$ are non-empty and disjoint, and $Y = U \cup V$. We need to show that $X$ is complete to $U$ and anti-complete to $V$. Let $X_0$ be the set of all vertices in $X$ that are complete to $U$ and anti-complete to $V$; by construction, $x_0 \in X_0$, and consequently, $X_0 \neq \emptyset$. We need to show that $X_0 = X$. Suppose otherwise. Then since $G[X]$ is connected, there exist adjacent $x \in X_0$ and $x' \in X \smallsetminus X_0$. Since $x \in X_0$, we know that $x$ is mixed on $Y$. Since $x' \in X \smallsetminus X_0$, we know that either $x'$ has a non-neighbor in $U$, or $x'$ has a neighbor in $V$. In either case, some vertex of $Y$ is mixed on $\{x,x'\}$, and so by~\ref{not-mixed-on-edge} $x$ is either complete or anticomplete to $Y$, which is a contradiction.
\end{proof}
\noindent 
We now need a definition. Given a graph $G$, a set $X \subseteq V_G$, and distinct vertices $u,v \in V_G \smallsetminus X$, we say that $u$ {\em dominates} $v$ in $X$ provided that every neighbor of $v$ in $X$ is also a neighbor of $u$. We are now ready to prove that every prime $\{P_5,\overline{P_5},C_5\}$-free graph that is not split admits a usable skew-partition. 
\begin{theorem} \label{preliminary-decomposition} Let $G$ be a prime $\{P_5,\overline{P_5},C_5\}$-free graph that is not split. Then $G$ admits a usable skew-partition. 
\end{theorem} 
\begin{proof} 
By \ref{skew-partition-two-nontrivial}, we know that $G$ admits a skew-partition $(X',Y')$ such that either $G[X']$ has at least two non-trivial components, or $G[Y']$ has at least two non-trivial anti-components. Since $G$ admits a usable skew-partition if and only if $\overline{G}$ does, we may assume that $G[X']$ contains at least two non-trivial components. 
\\
\\
Let $X \subseteq V_G$ be an inclusion-wise maximal set such that $(X,V_G \smallsetminus X)$ is a skew-partition of $G$, and $G[X]$ has at least two non-trivial components. Set $Y = V_G \smallsetminus X$. We claim that $(X,Y)$ satisfies (a) from the definition of a usable skew-partition. Let $(\{X_i\}_{i=1}^m,\{Y_j\}_{j=1}^n,S,K,\{S_j\}_{j=1}^n,\{K_i\}_{i=1}^m)$ be a decomposition of $(X,Y)$ in $G$. The fact that the sets $S_1,...,S_n$ are pairwise disjoint follows from \ref{not-mixed-on-two}, as does the fact that the sets $K_1,...,K_m$ are pairwise disjoint. 
\\
\\
It remains to show that every vertex in $Y$ has a neighbor in each of $X_1,...,X_m$. Suppose otherwise. Fix some $y \in Y$ such that $y$ is anti-complete to at least one of $X_1,...,X_m$; by symmetry, we may assume that $y$ is anti-complete to $X_1$. Since $G[X_1]$ is a non-trivial component of $G[X]$, since $G[X]$ has at least two non-trivial components, and since $y \notin X$ is anti-complete to $X_1$, we know that $G[X \cup \{y\}]$ has at least two non-trivial components. Now $G[Y \smallsetminus \{y\}]$ must be anti-connected, for otherwise, $X \cup \{y\}$ would contradict the maximality of $X$. Since $G[Y]$ is not anti-connected but $G[Y \smallsetminus \{y\}]$ is anti-connected, we know that $y$ is complete to $Y \smallsetminus \{y\}$. 
\\
\\
Since $G$ is prime, $X_1$ is not a homogeneous set in $G$. It follows that some vertex $y' \in V_G \smallsetminus X_1$ is mixed on $X_1$. Clearly, $y' \notin X$, and since $y$ is anti-complete to $X_1$, $y' \neq y$. Thus, $y' \in Y \smallsetminus \{y\}$. Now, $G[X_1]$ is connected, $G[Y \smallsetminus \{y\}]$ is anti-connected, $y$ is anti-complete to $X_1$ and complete to $Y \smallsetminus \{y\}$, and some vertex (namely $y'$) in $Y \smallsetminus \{y\}$ is mixed on $X_1$. By \ref{mixed-one-way}, we know that $Y \smallsetminus \{y\}$ is complete to $\Gamma_G(y') \cap X_1$ and anti-complete to $X_1 \smallsetminus \Gamma_G(y')$. Since $\Gamma_G(y') \cap X_1$ and $X_1 \smallsetminus \Gamma_G(y')$ are both non-empty (because $y'$ is mixed on $X_1$), it follows that every vertex in $Y \smallsetminus \{y\}$ is mixed on $X_1$. By \ref{not-mixed-on-two}, it follows that no vertex in $Y \smallsetminus \{y\}$ is mixed on any one of $X_2,...,X_m$. Since $m \geq 2$, since none of $X_2,...,X_m$ is a homogeneous set in $G$, and since (by \ref{not-mixed-on-two}) $y$ can be mixed on at most one of them, it follows that $m = 2$, and that $y$ is mixed on $X_2$. 
\\
\\
Now, we claim that every vertex in $Y \smallsetminus \{y\}$ dominates $y$ in $X$. Fix $\hat{y} \in Y \smallsetminus \{y\}$, and suppose that $\hat{y}$ does not dominate $y$ in $X$. Fix $\hat{x} \in X$ such that $y$ is adjacent to $\hat{x}$ but $\hat{y}$ is non-adjacent to $\hat{x}$. Since $y$ is anti-complete to $X_1$, we know that $\hat{x} \notin X_1$; since $G[X_1]$ is a component of $G[X]$, it follows that $\hat{x}$ is anti-complete to $X_1$. Since $\hat{y} \in Y \smallsetminus \{y\}$, we know that $\hat{y}$ is mixed on $X_1$; since $G[X_1]$ is connected, we know by \ref{mixed} that there exist adjacent vertices $x_1,x_1' \in X$ such that $\hat{y}$ is adjacent to $x_1$ and non-adjacent to $x_1'$. But now $\hat{x}-y-\hat{y}-x_1-x_1'$ is an induced four-edge path in $G$, which contradicts the assumption that $G$ is $P_5$-free. This proves that every vertex in $Y \smallsetminus \{y\}$ dominates $y$ in $X$. 
\\
\\
Let $Z = \{y\} \cup X_2 \cup (S \cap \Gamma_G(y))$. We claim that $Z$ is a homogeneous set in $G$. Clearly, $X \smallsetminus Z$ is anti-complete to $Z$. Next, we know that $y$ is complete to $Y \smallsetminus \{y\}$, and we proved above that every vertex in $Y \smallsetminus \{y\}$ dominates $y$ in $X$. Thus, $Y \smallsetminus \{y\}$ is complete to $\{y\} \cup (S \cap \Gamma_G(y))$, as well as to $X_2 \cap \Gamma_G(y)$. Now, we know that $y$ is mixed on $X_2$, and so $y$ has a neighbor in $X_2$; consequently, every vertex in $Y \smallsetminus \{y\}$ has a neighbor in $X_2$. Since no vertex in $Y \smallsetminus \{y\}$ is mixed on $X_2$, it follows that $Y \smallsetminus \{y\}$ is complete to $X_2$. Thus, $Y \smallsetminus \{y\}$ is complete to $Z$. This proves that $Z$ is a homogeneous set in $G$. Since $X_2 \subseteq Z$ and $Z \cap X_1 = \emptyset$, it follows that $2 \leq |Z| \leq |V_G|-2$, and consequently, $Z$ is a proper homogeneous set in $G$. But this contradicts the fact that $G$ is prime. 
\end{proof} 
\noindent 
The next two lemmas (\ref{not-mixed-on-other-half}, and \ref{skew-partition analysis}) examine the behavior of usable skew-partitions in $\{P_5,\overline{P_5},C_5\}$-free graphs. They will be used in the proof of \ref{skew-partition}, the main result of this section. 

\begin{theorem} \label{not-mixed-on-other-half} Let $G$ be a $\{P_5,\overline{P_5},C_5\}$-free graph, and let $(X,Y)$ be a usable skew-partition of $G$ with associated decomposition $(\{X_i\}_{i=1}^m,\{Y_j\}_{j=1}^n,S,K,\{S_j\}_{j=1}^n,$ $\{K_i\}_{i=1}^m)$. Then: 
\begin{itemize} 
\item if $(X,Y)$ satisfies (a) from the definition of a usable skew-partition, then no vertex in $X_1 \cup ... \cup X_m$ is mixed on any one of $Y_1,...,Y_n$; 
\item if $(X,Y)$ satisfies (b) from the definition of a usable skew-partition, then no vertex in $Y_1 \cup ... \cup Y_n$ is mixed on any one of $X_1,...,X_m$. 
\end{itemize} 
\end{theorem} 
\begin{proof} 
It suffices to prove the first claim, for then the second will follow by an analogous argument applied to $\overline{G}$. So assume that $(X,Y)$ satisfies (a) from the definition of a usable skew-partition. We need to show that no vertex in $X_1 \cup ... \cup X_m$ is mixed on any one of $Y_1,...,Y_n$. Suppose otherwise. By symmetry, we may assume that some vertex $x_1 \in X_1$ is mixed on $Y_1$. By~\ref{mixed}, there exist non-adjacent vertices $y_1,y_1' \in Y_1$ such that $x_1$ is adjacent to $y_1$ and non-adjacent to $y_1'$. Since $(X,Y)$ satisfies (a) from the definition of a usable skew-partition, it follows that for each $i \in \{1,2\}$, $y_1'$ has a neighbor $x_i' \in X_i$. Then $y_1'$ is mixed on $X_1$ (because $y_1'$ is adjacent to $x_1'$ and non-adjacent to $x_1$), and so by \ref{not-mixed-on-two}, $y_1'$ is not mixed on $X_2$. Since $y_1'$ has a neighbor (namely $x_2'$) in $X_2$, it follows that $y_1'$ is complete to $X_2$. Next, since $(X,Y)$ satisfies (a) from the definition of a usable skew-partition, $y_1$ has a neighbor $x_2 \in X_2$. Now $G[X_1]$ is connected, $G[y_1,y_1']$ is anti-connected, $x_2$ is anti-complete to $X_1$ and complete to $\{y_1,y_1'\}$, $y_1$ is non-adjacent to $y_1'$, and $x_1$ is adjacent to $y_1$ and non-adjacent to $y_1'$; thus, by \ref{not-mixed-on-edge}, $y_1'$ is anti-complete to $X_1$. But this is impossible because $y_1'$ has a neighbor (namely $x_1'$) in $X_1$. This completes the argument. 
\end{proof}

\begin{theorem} \label{skew-partition analysis} Let $G$ be a $\{P_5,\overline{P_5},C_5\}$-free graph, and let $(X,Y)$ be a usable skew-partition of $G$ with associated decomposition $(\{X_i\}_{i=1}^m,\{Y_j\}_{j=1}^n,S,K,\{S_j\}_{j=1}^n,$ $\{K_i\}_{i=1}^m)$. 
\begin{itemize} 
\item If the skew-partition $(X,Y)$ satisfies (a) from the definition of a usable skew-partition, then either $n = 0$, or the following hold: 
\begin{itemize} 
\item $S_1,...,S_n$ are non-empty, and 
\item there exists some $j \in \{1,...,n\}$ such that $S_j$ is anti-complete to $Y \smallsetminus (Y_j \cup K)$; 
\end{itemize} 
\item If the skew-partition $(X,Y)$ satisfies (b) from the definition of a usable skew-partition, then either $m = 0$, or the following hold: 
\begin{itemize} 
\item $K_1,...,K_m$ are non-empty, and 
\item there exists some $i \in \{1,...,m\}$ such that $K_i$ is complete to $X \smallsetminus (X_i \cup S)$. 
\end{itemize} 
\end{itemize} 
\end{theorem} 
\begin{proof} 
It suffices to prove the first statement, for then the second will follow by an analogous argument applied to $\overline{G}$. So suppose that the skew-partition $(X,Y)$ satisfies (a) from the definition of a usable skew-partition. If $n = 0$, then we are done; so suppose that $n \geq 1$. 
\\
\\
We first show that the sets $S_1,...,S_n$ are non-empty. By symmetry, it suffices to show that $S_1 \neq \emptyset$. By \ref{not-mixed-on-other-half}, no vertex of $X_1 \cup ...\cup X_m$ is mixed on $Y_1$. By construction, no vertex in $Y \smallsetminus Y_1$ is mixed on $Y_1$. Since $2 \leq |Y_1| \leq |V_G|-1$, and $Y_1$ is not a proper homogeneous set in $G$, it follows that some vertex in $S$ is mixed on $Y_1$, and therefore $S_1 \neq \emptyset$.
\\
\\
It remains to show that there exists some $i \in \{1,...,n\}$ such that $S_j$ is anti-complete to $Y \smallsetminus (Y_j \cup K)$. By what we just showed, the sets $S_1,...,S_n$ are non-empty, and since the skew-partition $(X,Y)$ is usable, the sets $S_1,...,S_n$ are pairwise disjoint. Now, let $\vec{H}$ be the directed graph with vertex-set $\{S_1,...,S_n\}$, and in which for all distinct $i,j \in \{1,...,n\}$, $(S_i,S_j)$ is an arc in $\vec{H}$ provided that some vertex in $S_i$ is complete to $Y_j$. 
\\
\\
We first prove the following: for all $t \geq 2$ and all pairwise distinct indices $i_1,...,i_t \in \{1,...,n\}$, if $S_{i_1}-...-S_{i_t}$ is a directed path in 
$\vec{H}$, then $S_{i_t}$ is anti-complete to $Y_{i_1} \cup ... \cup Y_{i_{t-1}}$ in $G$. We proceed by induction on $t$. 
\\
\\
For the base case, fix distinct $i_1,i_2 \in \{1,...,n\}$, and suppose that $S_{i_1}-S_{i_2}$ is a directed path in $\vec{H}$; then $(S_{i_1},S_{i_2})$ is an arc in $\vec{H}$, and consequently, some vertex $s_1 \in S_{i_1}$ is complete to $Y_{i_2}$. Now, suppose that $S_{i_2}$ is not anti-complete to $Y_{i_1}$; fix $s_2 \in S_{i_2}$ such that $s_2$ has a neighbor in $Y_{i_1}$. Then since no vertex in $S_{i_2}$ is mixed on $Y_{i_1}$, we know that $s_2$ is complete to $Y_{i_1}$. By definition, we know that $s_2$ is mixed on $Y_{i_2}$; by~\ref{mixed}, there exist non-adjacent vertices $y_2,y_2' \in Y_{i_2}$ such that $s_2$ is adjacent to $y_2$ and non-adjacent to $y_2'$. Next, by definition, $s_1$ is mixed on $Y_{i_1}$, and so there exists some $y_1' \in Y_{i_1}$ such that $s_1$ is non-adjacent to $y_1'$. Since $Y_{i_1}$ is complete to $Y_{i_2}$, we know that $y_1'$ is complete to $\{y_2,y_2'\}$, and since $S_{i_1} \cup S_{i_2}$ is a stable set, we know that $s_1$ is non-adjacent to $s_2$. But now $y_1'-s_1-s_2-y_2'-y_2$ is an induced house in $G$, which is a contradiction. This completes the base case. 
\\
\\
For the induction step, suppose that the claim holds for some $t \geq 2$; we need to show that it holds for $t+1$. Suppose that $i_1,...,i_t,i_{t+1} \in \{1,...,n\}$ are pairwise distinct and that $S_{i_1}-...-S_{i_t}-S_{i_{t+1}}$ is a directed path in $\vec{H}$. We need to show that $S_{i_{t+1}}$ is anti-complete to $Y_{i_1} \cup ... \cup Y_{i_t}$. By the induction hypothesis applied to the directed path $S_{i_2}-...-S_{i_t}-S_{i_{t+1}}$, we know that $S_{i_{t+1}}$ is anti-complete to $Y_{i_2} \cup... \cup Y_{i_t}$, and so we just need to show that $S_{i_{t+1}}$ is anti-complete to $Y_{i_1}$. Suppose otherwise. Then there exists some $s_{t+1} \in S_{i_{t+1}}$ such that $s_{t+1}$ has a neighbor in $Y_{i_1}$; since no vertex in $S_{i_{t+1}}$ is mixed on $Y_{i_1}$, this means that $s_{t+1}$ is complete to $Y_{i_1}$. Next, since $(S_{i_1},S_{i_2})$ is an arc in $\vec{H}$, we know that some vertex $s_1 \in S_{i_1}$ is complete to $Y_{i_2}$. By construction, $s_1$ is mixed on $Y_{i_1}$, and so by \ref{mixed}, we know that there exist non-adjacent $y_1,y_1' \in Y_{i_1}$ such that $s_1$ is adjacent to $y_1$ and non-adjacent to $y_1'$. Now, fix $y_2 \in Y_{i_2}$. Then $s_{t+1}$ is non-adjacent to $y_2$ and $s_1$ is adjacent to $y_2$. But now $y_1-y_1'-s_1-s_{t+1}-y_2$ is an induced house in $G$, which is a contradiction. This completes the induction. 
\\
\\
Now, let $S_{i_1}-...-S_{i_t}$ be a maximal directed path in $\vec{H}$. Set $j = i_t$, and fix $k \in \{1,...,n\} \smallsetminus \{j\}$; we need to show that $S_j$ is anti-complete to $Y_k$. If $k \in \{i_1,...,i_{t-1}\}$, then the result follows from what we just showed. So assume that $k \notin \{i_1,...,i_{t-1}\}$. Suppose that $S_j$ is not anti-complete to $Y_k$. Then some vertex $s_j \in S_j$ has a neighbor in $Y_k$; since $s_j$ is not mixed on $Y_k$, it follows that $s_j$ is complete to $Y_k$, and consequently, $(S_j,S_k)$ is an arc in $\vec{H}$. But now $S_{i_1}-...-S_{i_t}-S_k$ is a directed path in $\vec{H}$, contrary to the maximality of $S_{i_1}-...-S_{i_t}$. Thus, $Y_j$ is anti-complete to $Y_k$, and the result is immediate. 
\end{proof} 
\noindent 
We are finally ready to prove the main result of this section. 
\begin{theorem} \label{skew-partition} Let $G$ be a $\{P_5,\overline{P_5},C_5\}$-free graph. Then at least one of the following holds: 
\begin{itemize} 
\item[(1)] $G$ is a split graph; 
\item[(2)] $G$ contains a proper homogeneous set; 
\item[(3)] $G$ admits a skew-partition $(X,Y)$ with associated decomposition $(\{X_i\}_{i=1}^m,$ $\{Y_j\}_{j=1}^n,S,K,\{S_j\}_{j=1}^n,\{K_i\}_{i=1}^m)$ such that: 
\begin{itemize} 
\item[(3.1)] $m \geq 1$, 
\item[(3.2)] the sets $K_1,...,K_m$ are pairwise disjoint and non-empty, 
\item[(3.3)] for all $i \in \{1,...,m\}$, no vertex in $Y \smallsetminus K_i$ is mixed on $X_i$, 
\item[(3.4)] for all $i \in \{1,..,m\}$, at least two vertices in $V_G \smallsetminus (X_i \cup S)$ are anti-complete to $X_i$, 
\item[(3.5)] there exists some $i \in \{1,..,m\}$ such that $K_i$ is complete to $X \smallsetminus (X_i \cup S)$; 
\end{itemize} 
\item[(4)] $G$ admits a skew-partition $(X,Y)$ with associated decomposition $(\{X_i\}_{i=1}^m,$ $\{Y_j\}_{j=1}^n,S,K,\{S_j\}_{j=1}^n,\{K_i\}_{i=1}^m)$ such that: 
\begin{itemize} 
\item[(4.1)] $n \geq 1$, 
\item[(4.2)] the sets $S_1,...,S_n$ are pairwise disjoint and non-empty, 
\item[(4.3)] for all $j \in \{1,...,n\}$, no vertex in $X \smallsetminus S_j$ is mixed on $Y_j$, 
\item[(4.4)] for all $j \in \{1,...,n\}$, at least two vertices in $V_G \smallsetminus (Y_j \cup K)$ are complete to $Y_j$, 
\item[(4.5)] there exists some $j \in \{1,...,m\}$ such that $S_j$ is anti-complete to $Y \smallsetminus (Y_j \cup K)$. 
\end{itemize} 
\end{itemize} 
\end{theorem} 
\begin{proof} 
If $G$ is a split graph, or if $G$ contains a proper homogeneous set, then we are done. So assume that $G$ is a prime graph, and that $G$ is not split. Then by \ref{preliminary-decomposition}, $G$ admits a usable skew-partition. Let $(X,Y)$ be a usable skew-partition of $G$, and let $(\{X_i\}_{i=1}^m,\{Y_j\}_{j=1}^n,S,K,\{S_j\}_{j=1}^n,\{K_i\}_{i=1}^m)$ be its associated decomposition. By definition, if $(X,Y)$ satisfies (a) from the definition of a usable skew-partition, then $m \geq 2$, and if $(X,Y)$ satisfies (b) from the definition of a usable skew-partition, then $n \geq 2$. Thus, there are four cases to consider: 
\begin{itemize} 
\item $(X,Y)$ satisfies (a) from the definition of a skew-partition, $m \geq 2$, and $n = 0$; 
\item $(X,Y)$ satisfies (a) from the definition of a skew-partition, $m \geq 2$, and $n \geq 1$; 
\item $(X,Y)$ satisfies (b) from the definition of a skew-partition, $n \geq 2$, and $m = 0$; 
\item $(X,Y)$ satisfies (b) from the definition of a skew-partition, $n \geq 2$, and $m \geq 1$. 
\end{itemize} 
We claim that in the first and fourth case, (3) holds, and in the second and third case, (4) holds. But note that given the symmetry between $G$ and $\overline{G}$, we need only consider the first and fourth case. 
\\
\\
Suppose first that $(X,Y)$ satisfies (a) from the definition of a skew-partition, $m \geq 2$, and $n = 0$. Then (3.1) is immediate. We next prove (3.2). Since each of $X_1,...,X_m$ contains at least two vertices, and since $G$ is prime, we know that for all $i \in \{1,...,m\}$, some vertex in $V_G \smallsetminus X_i$ is mixed on $X_i$; clearly, this vertex cannot lie in $X$, and so it must lie in $Y$. Since $Y = K$, it follows that for all $i \in \{1,...,m\}$, some vertex in $K$ is mixed on $X_i$, and by definition, this vertex lies in $K_i$. This proves that each of $K_1,...,K_m$ is non-empty. The fact that $K_1,...,K_m$ are pairwise disjoint follows from the definition of a usable skew-partition. This proves (3.2). For (3.3), we first observe that since $n = 0$, we have that $Y = K$. By the definition of $K_1,...,K_m$, we know that for all $i \in \{1,...,m\}$, no vertex in $K \smallsetminus K_i$ is mixed on $X_i$. This proves (3.3). Next, $X_2$ is anti-complete to $X_1$ and $|X_2| \geq 2$, and (3.4) follows. For (3.5), we note that since $(X,Y)$ satisfies (a) from the definition of a usable skew-partition, every vertex in $Y$ has a neighbor in each of $X_1,...,X_m$. Now, fix an arbitrary $i \in \{1,...,m\}$. Then each vertex in $K_i$ has a neighbor in each of $X_1,...,X_m$. Since the sets $K_1,...,K_m$ are pairwise disjoint, it follows that no vertex in $K_i$ is mixed on any one of $X_1,...,X_{i-1},X_{i+1},...,X_m$, and consequently, every vertex in $K_i$ is complete to each of $X_1,...,X_{i-1},X_{i+1},...,X_m$. Thus, $K_i$ is complete to $X \smallsetminus (X_i \cup S)$, and (3.5) follows. 
\\
\\
Suppose now that $(X,Y)$ satisfies (b) from the definition of a skew-partition, $n \geq 2$, and $m \geq 1$. Then (3.1) is immediate. The fact that the sets $K_1,...,K_m$ are non-empty follows from \ref{skew-partition analysis}, and the fact that they are pairwise disjoint follows from the definition of a usable skew-partition; this proves (3.2). (3.3) follows from \ref{not-mixed-on-other-half}. For (3.4), fix $i \in \{1,...,m\}$, and fix $x_i \in X_i$. Since $(X,Y)$ satisfies (a) from the definition of a usable skew-partition, we know that $x_i$ has a non-neighbor in each of $Y_1,...,Y_n$; since $n \geq 2$, there exist $y_1 \in Y_1$ and $y_2 \in Y_2$ such that $x_i$ is non-adjacent to both $y_1$ and $y_2$. Thus, $y_1$ and $y_2$ both have a non-neighbor (namely $x_i$) in $X_i$. By \ref{not-mixed-on-other-half}, neither $y_1$ nor $y_2$ is mixed on $X_i$; thus, $y_1$ and $y_2$ are both anti-complete to $X_1$. This proves (3.4). Finally, (3.5) follows from \ref{skew-partition analysis}. 
\end{proof}

\section{Split graph divide} \label{section:split-graph-divide}

Given a graph $G$ and pairwise disjoint (possibly empty) sets $A,B,C,L,T \subseteq V_G$, we say that $(A,B,C,L,T)$ is a {\em split graph divide} of $G$ provided that the following hold: 
\begin{itemize} 
\item $V_G = A \cup B \cup C \cup L \cup T$; 
\item $|A| \geq 2$; 
\item $|C| \geq 2$; 
\item $L$ is a non-empty clique;
\item $T$ is a (possibly empty) stable set; 
\item every vertex of $L$ is mixed on $A$; 
\item $A$ is complete to $B$ and anti-complete to $C \cup T$; 
\item $L$ is complete to $B \cup C$; 
\item $T$ is anti-complete to $C$. 
\end{itemize} 
\noindent 
We say that a graph $G$ {\em admits a split graph divide} provided that there exist sets $A,B,C,L,T \subseteq V_G$ such that $(A,B,C,L,T)$ is a split graph divide of $G$. 
\\
\\
Split graph divide can be thought of as a relaxation of the homogeneous set decomposition. A set $X \subseteq V_G$ is a homogeneous set in $G$ if no vertex in $V_G \smallsetminus X$ is mixed on $X$. In the case of the split graph divide, there are vertices that are mixed on the set $A$, but they all lie in the clique $L$, and adjacency between $L$ and the rest of the graph is heavily restricted. 
\\
\\
We now use \ref{skew-partition} to prove another decomposition theorem for $\{P_5,\overline{P_5},C_5\}$-free graphs, which is the main result of this section. 
\begin{theorem} \label{decomposition} Let $G$ be a $\{P_5,\overline{P_5},C_5\}$-free graph. Then at least one of the following holds: 
\begin{itemize} 
\item $G$ is a split graph; 
\item $G$ contains a proper homogeneous set; 
\item at least one of $G$ and $\overline{G}$ admits a split graph divide. 
\end{itemize} 
\end{theorem} 
\begin{proof} 
We may assume that $G$ is prime and that it is not a split graph, for otherwise we are done. Then by \ref{skew-partition}, $G$ admits a skew-partition $(X,Y)$ with associated decomposition $(\{X_i\}_{i=1}^m,$ $\{Y_j\}_{j=1}^n,S,K,\{S_j\}_{j=1}^n,\{K_i\}_{i=1}^m)$ that satisfies either (3.1)-(3.5) or (4.1)-(4.5) from \ref{skew-partition}. We claim that if $(X,Y)$ satisfies (3.1)-(3.5) from \ref{skew-partition}, then $G$ admits a split graph divide, and if $(X,Y)$ satisfies (4.1)-(4.5) from \ref{skew-partition}, then $\overline{G}$ admits a split graph divide. But it suffices to prove the first claim, for then the second will follow by an analogous argument applied to $\overline{G}$. 
\\
\\
So suppose that $(X,Y)$ satisfies (3.1)-(3.5) from \ref{skew-partition}. By (3.5) and by symmetry, we may assume that $K_1$ is complete to $X \smallsetminus (X_1 \cup S)$, that is, that $K_1$ is complete to $X_2 \cup ... \cup X_m$. We now construct sets $A,B,C,L,T$ as follows: 
\begin{itemize} 
\item let $A = X_1$; 
\item let $B$ be the set of all vertices in $Y$ that are complete to $X_1$; 
\item let $C$ be the union of the following three sets: 
\begin{itemize} 
\item $X_2 \cup ... \cup X_m$, 
\item the set of all vertices in $Y$ that are anti-complete to $X_1$, 
\item the set of all vertices in $S$ that are complete to $K_1$; 
\end{itemize} 
\item let $L = K_1$; 
\item let $T$ be the set of all vertices in $S$ that have a non-neighbor in $K_1$. 
\end{itemize} 
\noindent 
First, it is clear that the sets $A,B,C,L,T$ are pairwise disjoint. Next, it is clear that $X \cup K_1 \subseteq A \cup B \cup C \cup L \cup T$, and it is also clear that all vertices in $Y$ that are not mixed on $X_1$ lie in $A \cup B \cup C \cup L \cup T$. But since by (3.3), no vertex in $Y \smallsetminus K_1$ is mixed on $X_1$, it follows that $Y \smallsetminus K_1 \subseteq A \cup B \cup C \cup L \cup T$. This proves that $V_G = A \cup B \cup C \cup L \cup T$. It is immediate by construction that $G[A]$ is connected and that $|A| \geq 2$. The fact that $|C| \geq 2$ follows from (3.4). By construction, $L$ is a clique, and by (3.2), $L$ is non-empty. By construction, $T$ is a stable set, every vertex of $L$ is mixed on $A$, and $A$ is complete to $B$ and anti-complete to $C \cup T$. It remains to prove that $L$ is complete to $B \cup C$, and that $T$ is anti-complete to $C$. 
\\
\\
First, we know by construction that $K_1$ is complete to $Y \smallsetminus K_1$. Thus, to show that $L$ is complete to $B \cup C$, we just need to show that $K_1$ is complete to $X_2 \cup ... \cup X_m$. But this follows from the choice of $K_1$. 
\\
\\
It remains to show that $T$ is anti-complete to $C$. This means that we have to show that $T$ is anti-complete to each of the following three sets: 
\begin{itemize} 
\item $X_2 \cup ... \cup X_m$, 
\item the set of all vertices in $Y$ that are anti-complete to $X_1$, 
\item the set of all vertices in $S$ that are complete to $K_1$; 
\end{itemize} 
\noindent 
It is clear that $T$ is anti-complete to the first and the third set, and we just need to prove that $T$ is anti-complete to the second of the three sets above. Suppose otherwise. Fix adjacent $s \in T$ and $y \in Y$ such that $y$ is anti-complete to $X_1$. Since $s \in T$, we know that $s$ has a non-neighbor $k_1 \in K_1$. Since $k_1 \in K_1$, we know that $k_1$ is mixed on $X_1$. Since $G[X_1]$ is connected, we know by \ref{mixed} that there exist adjacent vertices $x_1,x_1' \in X_1$ such that $k_1$ is adjacent to $x_1$ and non-adjacent to $x_1'$. But now $s-y-k_1-x_1-x_1'$ is an induced four-edge path in $G$, which is impossible. Thus, $T$ is anti-complete to $C$. This completes the argument. 
\end{proof}

\section{Split graph unification} \label{section:split-graph-unification} 

In this section we define a composition operation that ``reverses'' the split graph divide decomposition. Let $A,B,C,L,T$ be pairwise disjoint sets, and assume that $A$ and $C$ are non-empty. Let $a,c$ be distinct vertices such that $a,c \notin A \cup B \cup C \cup L \cup T$. Let $G_1$ be a graph with vertex-set $V_{G_1} = A \cup L \cup T \cup \{c\}$, and adjacency as follows: 
\begin{itemize} 
\item $L$ is a (possibly empty) clique; 
\item $T$ is a (possibly empty) stable set; 
\item $A$ is anti-complete to $T$; 
\item $c$ is complete to $L$ and anti-complete to $A \cup T$. 
\end{itemize} 
Let $G_2$ be a graph with vertex-set $V_{G_2} = B \cup C \cup L \cup T \cup \{a\}$, and adjacency as follows: 
\begin{itemize} 
\item $G_2[L \cup T] = G_1[L \cup T]$; 
\item $T$ is anti-complete to $C$; 
\item $L$ is complete to $B \cup C$; 
\item $a$ is complete to $B$ and anti-complete to $C \cup L \cup T$. 
\end{itemize} 
Under these circumstances, we say that $(G_1,G_2)$ is a {\em composable pair}. 
The {\em split graph unification} of a composable pair $(G_1,G_2)$ is the graph $G$ with vertex-set $A \cup B \cup C \cup L \cup T$ such that: 
\begin{itemize} 
\item $G[A \cup L \cup T] = G_1 \smallsetminus c$; 
\item $G[B \cup C \cup L \cup T] = G_2 \smallsetminus a$; 
\item $A$ is complete to $B$ and anti-complete to $C$ in $G$. 
\end{itemize} 
\noindent 
Thus to obtain $G$ from $G_1$ and $G_2$, we ``glue'' $G_1$ and $G_2$ along their common induced subgraph $G_1[L \cup T]=G_2[L \cup T]$, and this induced subgraph is a split graph (hence the name of the operation). 
\\
\\
We say that a graph $G$ is {\em obtained by split graph unification from graphs with property $P$} provided that there exists a composable pair $(G_1,G_2)$ such that $G_1$ and $G_2$ both have property $P$, and $G$ is the split graph unification of $(G_1,G_2)$. We say that $G$ is {\em obtained by split graph unification in the complement from graphs with property $P$} provided that $\overline{G}$ is obtained by split graph unification from graphs with property $P$. 
\\
\\
We now prove that every graph that admits a split graph divide is obtained by split graph unification from smaller graphs. 
\begin{theorem} \label{divuni} If a graph admits a split graph divide, then it is obtained from a composable pair of smaller graphs by split graph unification.
\end{theorem}
\begin{proof}
Let $G$ be a graph that admits a split graph divide, and let $(A,B,C,L,T)$ be a split graph divide of $G$. Let $G_1$ be the graph obtained from $G[A\cup L \cup T]$ by adding a new vertex $c$ complete to $L$ and anticomplete to $A \cup T$. Since $|C| \geq 2$, we know that $|V_{G_1}|<|V_G|$. Let $G_2$ be obtained from $G[B \cup C \cup L \cup T]$ by adding a new vertex $a$ complete to $B$ and anticomplete to $C \cup L \cup T$. Since $|A| \geq 2$, we know that $|V_{G_2}|<|V_G|$. Now $(G_1,G_2)$ is a composable pair, and $G$ is obtained from it by split graph unification. This completes the argument.
\end{proof} 
\noindent 
Note that in the proof of~\ref{divuni}, $G_1$ is  obtained from $G$ by first deleting $B$, and then ``shrinking'' $C$ to a vertex $c$. On the other hand, $G_2$ is obtained from $G$ by first deleting all the edges between $A$ and $L$ and then ``shrinking'' $A$ to a vertex $a$. Thus,  $G_1$ is (isomorphic to) an induced subgraph of $G$, but $G_2$ need not be. 
\\
\\
Split graph unification can be thought of as generalized substitution. Indeed, we obtain the graph $G$ from $G_1$ and $G_2$ by first substituting $G_1[A]$ for $a$ in $G_2$, and then ``reconstructing'' the adjacency between $A$ and $L$ in $G$ using the adjacency between $A$ and $L$ in $G_1$. We include $T$ and $c$ in $G_1$ in order to ensure that split graph unification preserves the property of being $\{P_5,\overline{P_5},C_5\}$-free. In fact, we prove something stronger than this: split graph unification preserves the (individual) properties of being $P_5$-free, $\overline{P_5}$-free, and $C_5$-free. 
\begin{theorem} \label{composition-free} Let $(G_1,G_2)$ be a composable pair, and let $G$ be the split graph unification of $(G_1,G_2)$. Then: 
\begin{itemize} 
\item if $G_1$ and $G_2$ are $P_5$-free, then $G$ is $P_5$-free; 
\item if $G_1$ and $G_2$ are $\overline{P_5}$-free, then $G$ is $\overline{P_5}$-free; 
\item if $G_1$ and $G_2$ are $C_5$-free, then $G$ is $C_5$-free. 
\end{itemize} 
\end{theorem} 
\begin{proof} 
Let $H \in \{P_5,\overline{P_5},C_5\}$, and suppose that $G_1$ and $G_2$ are $H$-free. We need to show that $G$ is $H$-free. Suppose otherwise. Fix $W \subseteq V_G$ such that $G[W] \cong H$. Let $A,B,C,L,T,a,c$ be as in the definition of a composable pair. Since $H$ is prime, the class of $H$-free graphs is closed under substitution. Now, since $G_1$ and $G_2$ are $H$-free, and since $G \smallsetminus B$ is obtained by substituting $G_2[C]$ for $c$ in $G_1$, we know that $G \smallsetminus B$ is $H$-free. Thus, $W \cap B \neq \emptyset$. Next, since $G \smallsetminus A$ is an induced subgraph of $G_2$, and $G_2$ is $H$-free, we know that $W \cap A \neq \emptyset$. Since $G_1$ and $G_2$ are $H$-free, and $G \smallsetminus L$ is obtained by substituting $G_1[A]$ for $a$ in $G_2 \smallsetminus L$, we know that $W \cap L \neq \emptyset$. Since $B$ is complete to $A \cup L$ in $G$, and since $W$ intersects both $B$ and $A \cup L$ in $G$, we know that $G[W \cap (A \cup B \cup L)]$ is not anti-connected. Since $H$ is anti-connected, we know that $W \not\subseteq A \cup B \cup L$; thus, $W \cap (C \cup T) \neq \emptyset$. Since $W$ intersects each of $A$, $B$, $L$, and $C \cup T$, and since $|W| = 5$, we know that $1 \leq |W \cap A| \leq 2$. 
\\
\\
Suppose first that $|W \cap A| = 2$; set $W \cap A = \{a_1,a_2\}$. Since $W$ intersects each of $B$, $L$, and $C \cup T$, and since $|W| = 5$, it follows that $|W \cap B| = |W \cap L| = |W \cap (C \cup T)| = 1$. Set $W \cap B = \{b\}$, $W \cap L = \{l\}$, and $W \cap (C \cup T) = \{w\}$. Then $W = \{a_1,a_2,b,l,w\}$. Since $G[W]$ is prime, $\{a_1,a_2\}$ cannot be a proper homogeneous set in $G[W]$; since $b$ is complete to $A$ and $w$ is anti-complete to $A$ in $G$, we know that $l$ is mixed on $\{a_1,a_2\}$. By symmetry, we may assume that $l$ is adjacent to $a_1$ and non-adjacent to $a_2$. Then $\{a_1,b,l\}$ is a triangle in $G[W]$, and consequently, $H$ is a house. Thus, $a_2$ must have at least two neighbors in $G[W]$. Since $a_2$ is non-adjacent to $l$ (by assumption) and to $w$ (because $A$ is anti-complete to $C \cup T$), it follows that $a_2$ is adjacent to $a_1$ and $b$. But now $\{a_1,a_2,b\}$ and $\{a_1,b,l\}$ are two distinct triangles in the house $G[W]$, contrary to the fact that a house has only one triangle. 
\\
\\
It remains to consider the case when $|W \cap A| = 1$. Set $W \cap A = \{\hat{a}\}$. If $\hat{a}$ is anti-complete to $W \cap L$, then $G[W]$ is an induced subgraph of $G_2$, contrary to the fact that $G_2$ is $H$-free. Thus, $\hat{a}$ has a neighbor $l \in L$. Since $W \cap B \neq \emptyset$, there exists some $b \in W \cap B$. Since $B$ is complete to $A \cup L$ in $G$, we know that $\{\hat{a},b,l\}$ is a triangle in $G[W]$, and consequently, $H$ is a house. Now, suppose that $|W \cap B| \geq 2$, and fix some $b' \in (W \cap B) \smallsetminus \{b'\}$. But then $\{\hat{a},b,l\}$ and $\{\hat{a},b',l\}$ are distinct triangles in the house $G[W]$, contrary to the fact that a house contains only one triangle. Thus, $W \cap B = \{b\}$. Next, suppose that $|W \cap L| \geq 2$, and fix some $l' \in W \cap L$. But then since $L$ is a clique in $G$, and since $B$ is complete to $L$ in $G$, we know that $\{b,l,l'\}$ is a triangle in $G[W]$, and so $G[W]$ contains at least two triangles (namely $\{\hat{a},b,l\}$ and $\{b,l,l'\}$), contrary to the fact that $G[W]$ is a house. Thus, $W \cap L = \{l\}$. It now follows that $|W \cap (C \cup T)| = 2$; set $W \cap (C \cup T) = \{c_1,c_2\}$. Since $\{\hat{a},b,l\}$ is the unique triangle of the house $G[W]$, we know that $c_1,c_2$ is an edge; since $T$ is a stable set in $G$, and since $C$ is anti-complete to $T$ in $G$, this implies that $c_1,c_2 \in C$. Since $c_1c_2$ is an edge in $G[W]$, and since $C$ is complete to $L$ in $G$, we know that $\{c_1,c_2,l\}$ is a triangle in $G[W]$. But now the house $G[W]$ contains two distinct triangles (namely $\{\hat{a},b,l\}$ and $\{c_1,c_2,l\}$), which is impossible. This completes the argument. 
\end{proof} 
\noindent 
We now prove a partial converse of \ref{composition-free}. 
\begin{theorem} \label{components-free} Let $(G_1,G_2)$ be a composable pair, let $A,B,C,L,T,a,c$ be as in the definition of a composable pair, and let $G$ be the split graph unification of $(G_1,G_2)$. Then: 
\begin{itemize} 
\item if $G$ is $P_5$-free, then $G_1$ and $G_2$ are $P_5$-free; 
\item if $G$ is $\overline{P_5}$-free, and every vertex of $L$ has a non-neighbor in $A$ in $G$, then $G_1$ and $G_2$ are $\overline{P_5}$-free; 
\item if $G$ is $C_5$-free, then $G_1$ and $G_2$ are $C_5$-free. 
\end{itemize} 
\end{theorem} 
\begin{proof} 
Let $H \in \{P_5,\overline{P_5},C_5\}$, and suppose that $G$ is $H$-free. If $H = \overline{P_5}$, we additionally assume that every vertex of $L$ has a non-neighbor in $A$ in $G$ (and consequently, in $G_1$ as well). We need to show that $G_1$ and $G_2$ are both $H$-free. Clearly, $G_1$ is an induced subgraph of $G$, and consequently, $G_1$ is $H$-free. It remains to show that $G_2$ is $H$-free. Suppose otherwise. Fix some $W \subseteq V_{G_2}$ such that $G[W] \cong H$. 
\\
\\
First, we claim that $a \in W$, and that $W$ intersects each of $B$, $L$, and $C \cup T$. Since $G_2 \smallsetminus a$ is an induced subgraph of $G$, and $G$ is $H$-free, we know that $a \in W$. Next, since $|W| = 5$, since $G[W]$ is connected, since $a \in W$, and since all neighbors of $a$ in $G_2$ lie in $B$, we know that $W \cap B \neq \emptyset$. Since $G_2 \smallsetminus L$ is (isomorphic to) an induced subgraph of $G$, we know that $W \cap L \neq \emptyset$. Since $\{a\} \cup L$ is complete to $B$ in $G_2$, and since $W$ intersects both $\{a\} \cup L$ and $B$, we know that $G[W \cap (\{a\} \cup B \cup L)]$ is not anti-connected. Since $G[W]$ is anti-connected, it follows that $W \not\subseteq \{a\} \cup B \cup L$. Consequently, $W \cap (C \cup T) \neq \emptyset$. This proves our claim. 
\\
\\
Now, we deal with the following two cases separately: when $H = \overline{P_5}$, and when $H \in \{P_5,C_5\}$. 
\\
\\
Suppose first that $H = \overline{P_5}$. Then by assumption, every vertex in $L$ has a non-neighbor in $A$ in $G$, and it follows that for all $l \in L$, $G_2 \smallsetminus (L \smallsetminus \{l\})$ is (isomorphic to) an induced subgraph of $G$. Thus, $|W \cap L| \geq 2$. Since $|W| = 5$, it follows that $a \in W$, $|W \cap L| = 2$, and $|W \cap B| = |W \cap (C \cup T)| = 1$. Since all neighbors of $a$ in $G_2$ lie in $B$, and since $|W \cap B| = 1$, it follows that $a$ has at most one neighbor in $G[W]$. But this is impossible because $G[W]$ is a house, and every vertex of a house is of degree at least two. 
\\
\\
It remains to consider the case when $H \in \{P_5,C_5\}$. Fix $b \in W \cap B$ and $l \in W \cap L$. Since $a$ is complete to $B$ and anti-complete to $L$ in $G_2$, and since $B$ is complete to $L$ in $G_2$, we know that $a-b-l$ is an induced path in $G_2[W]$. 
\\
\\
We claim that $W \cap B = \{b\}$. Suppose otherwise; fix $b' \in (W \cap B) \smallsetminus \{b\}$. Then $a-b-l-b'-a$ is a (not necessarily induced) square in $G[W]$, which is impossible because $G[W]$ is either a four-edge path or a pentagon. Thus, $W \cap B = \{b\}$. Next, we claim that $W \cap L = \{l\}$. Suppose otherwise; fix $l' \in (W \cap L) \smallsetminus \{l\}$. Since $L$ is a clique in $G_2$, and since $B$ is complete to $L$ in $G_2$, it follows that $\{b,l,l'\}$ is a triangle in $G[W]$. But this is impossible since $G[W]$ is either a four-edge path or a pentagon. This proves that $W \cap L = \{l\}$. 
\\
\\
Since all neighbors of $a$ in $G_2$ lie in $B$, and since $W \cap B = \{b\}$, we know that $a$ is of degree at most one in $G[W]$. Consequently, $G[W]$ is a four-edge path; in fact, the four-edge path $G[W]$ must be of the form $a-b-l-c_1-c_2$, where $c_1,c_2 \in W \cap (C \cup T)$. Since $c_1c_2$ is an edge, and since $T$ is a stable set that is anti-complete to $C$ in $G_2$, it follows that $c_1,c_2 \in C$. But $C$ is complete to $L$ in $G_2$, and so $\{c_1,c_2,l\}$ is a triangle in $G_2[W]$, which is impossible since $G_2[W]$ is a four-edge path. This completes the argument. 
\end{proof} 
\noindent 
We remark that the additional assumption in the second statement of \ref{components-free} is needed because of the following example. Let $G_1$ be a path $a_1-l-c$, and let  $G_2$ be a house $b_1-b_2-c_1-a-l$. Set $A = \{a_1\}$, $B = \{b_1,b_2\}$, $C = \{c_1\}$, $L = \{l\}$, and $T = \emptyset$. With this setup, $(G_1,G_2)$ is easily seen to be a composable pair. Let  $G$ be the split graph unification of $(G_1,G_2)$. It is easy to check that  $G$ is $\overline{P_5}$-free, even though $G_2$ is a house. 
\\
\\
We complete this section with a strengthening of \ref{divuni}, which we will need in section \ref{section:main-thm}. 
\begin{theorem} \label{divuni-free} Let $H \in \{P_5,\overline{P_5},C_5\}$, and let $G$ be an $H$-free graph that admits a split graph divide. Then $G$ is obtained from a composable pair of smaller $H$-free graphs by split graph unification. 
\end{theorem}
\begin{proof}
Let $(A,B,C,L,T)$ be a split graph divide of $G$, and let $G_1$ and $G_2$ be constructed as in the proof of \ref{divuni}. As shown in the proof of \ref{divuni}, $G_1$ and $G_2$ are both smaller than $G$, $(G_1,G_2)$ is a composable pair, and $G$ is the split graph unification of $(G_1,G_2)$. Further, since $(A,B,C,D,L,T)$ is a split graph divide of $G$, we know that every vertex in $L$ is mixed on $A$ in $G$, and in particular, that every vertex of $L$ has a non-neighbor in $A$ in $G$. By \ref{components-free} then, $G_1$ and $G_2$ are both $H$-free. 
\end{proof} 

\section{The main theorem} \label{section:main-thm} 

In this section, we use \ref{Fouquet} and the results of sections \ref{section:split-graph-divide} and \ref{section:split-graph-unification} to prove \ref{P5-P5c}, the main theorem of this paper. 
\begin{theorem} \label{P5-P5c} A graph $G$ is $\{P_5,\overline{P_5}\}$-free if and only if at least one of the following holds: 
\begin{itemize} 
\item $G$ is a split graph; 
\item $G$ is a pentagon; 
\item $G$ is obtained by substitution from smaller $\{P_5,\overline{P_5}\}$-free graphs; 
\item $G$ or $\overline{G}$ is obtained by split graph unification from smaller $\{P_5,\overline{P_5}\}$-free graphs. 
\end{itemize} 
\end{theorem}
\begin{proof}
We first prove the ``if'' part. If $G$ is a split graph or a pentagon, then it is clear that $G$ is $\{P_5,\overline{P_5}\}$-free. Since both $P_5$ and $\overline{P_5}$ are  prime, we know that the class of $\{P_5,\overline{P_5}\}$-free graphs is closed under substitution, and consequently, any graph obtained by substitution from smaller $\{P_5,\overline{P_5}\}$-free graphs is $\{P_5,\overline{P_5}\}$-free. Finally, if $G$ or $\overline{G}$ is obtained by split graph unification from smaller $\{P_5,\overline{P_5}\}$-free graphs, then the fact that $G$ is $\{P_5,\overline{P_5}\}$-free follows from \ref{composition-free} and from the fact that the complement of a $\{P_5,\overline{P_5}\}$-free graph is again $\{P_5,\overline{P_5}\}$-free. 
\\
\\
For the ``only if'' part, suppose that $G$ is a $\{P_5,\overline{P_5}\}$-free graph. We may assume that $G$ is prime, for otherwise, $G$ is obtained by substitution from smaller $\{P_5,\overline{P_5}\}$-free graphs, and we are done. If some induced subgraph of $G$ is isomorphic to the pentagon, then by~\ref{Fouquet}, $G$ is  a pentagon, and again we are done. Thus we may assume that $G$ is $\{P_5,\overline{P_5},C_5\}$-free. By \ref{decomposition}, we know that either $G$ is a split graph, or one of $G$ and $\overline{G}$ admits a split graph divide. In the former case, we are done. In the latter case, \ref{divuni-free} implies that $G$ or $\overline{G}$ is the split graph unification of a composable pair of smaller $\{P_5,\overline{P_5},C_5\}$-free graphs, and again we are done. 
\end{proof} 
\noindent 
As an immediate corollary of \ref{P5-P5c}, we have the following. 
\begin{theorem}\label{cor}
A graph is $\{P_5,\overline{P_5}\}$-free if and only if it is obtained from pentagons and split graphs by repeated substitutions, split graph unifications, and split graph unifications in the complement.
\end{theorem}
\noindent 
Finally, a proof analogous to the proof of \ref{P5-P5c} (but without the use of \ref{Fouquet}) yields the following result for $\{P_5,\overline{P_5},C_5\}$-free graphs. 
\begin{theorem} \label{P5-P5c-C5} A graph $G$ is $\{P_5,\overline{P_5},C_5\}$-free if and only if at least one of the following holds: 
\begin{itemize}
\item $G$ is a split graph; 
\item $G$ is obtained by substitution from smaller $\{P_5,\overline{P_5},C_5\}$-free graphs; 
\item $G$ or $\overline{G}$ is obtained by split graph unification from smaller $\{P_5,\overline{P_5},C_5\}$-free graphs. 
\end{itemize} 
\end{theorem}

\section{Acknowledgment} 

We would like to thank Ryan Hayward, James Nastos, Paul Seymour, and Yori Zwols for many useful discussions.

\end{document}